\documentclass[11pt,reqno]{amsart}
\usepackage{amssymb,latexsym}
\usepackage{graphicx}
\usepackage{fancyhdr}

\usepackage[all]{xy}
\usepackage{enumerate}
\theoremstyle{plain}
\numberwithin{equation}{section}
\newtheorem{thm}{Theorem}[section]
\newtheorem{theorem}[thm]{Theorem}

\newtheorem{definition}[thm]{Definition}
\newtheorem{proposition}[thm]{Proposition}
\newtheorem*{theorem*}{Theorem}

\begin{document}

\setcounter{page}{1}

\title[]{On $G$-Derivations of Lie-Yamaguti Algebras}
\author[Sahoo]{Aroonima Sahoo}
\address{National Institute of Technology Rourkela \\
         Odisha-769008 \\
                India}
\email{aroonima.sahoo05@gmail.com}

\author[Khuntia]{Tofan Kumar Khuntia}
\address{National Institute of Technology Rourkela \\
         Odisha-769008 \\
                India}
\email{tofankhuntia27@gmail.com}

\author[Pati]{Kishor Chandra Pati}
\address{National Institute of Technology Rourkela \\
         Odisha-769008 \\
                India}
\email{kcpati@nitrkl.ac.in}
\subjclass[2020]{17A30, 17A36, 17A40.}
\keywords{Lie-Yamaguti algebra; Derivation; Quasi-derivation }

\begin{abstract}
\label{abstract}
 This paper primarily deals with the study of $G$-derivations associated with Lie-Yamaguti algebras. Taking $G$ as an automorphism group, the concept of $G$-derivations, which is a derivation under both the bilinear and trilinear operations is defined for Lie-Yamaguti algebras. Then some important properties of $G$-derivations are studied along with their relationship with other generalized derivations of Lie-Yamaguti algebras.
\end{abstract}

\maketitle

\section{Introduction}
\label{intro}

 Lie-Yamaguti algebras are special type of algebras which entertain the properties of Lie algebras as well as Lie triple systems. Jacobson \cite{jac1989} in 1989 firstly defined a submodule of an associative algebra known as a Lie triple system that is closed under the iterated commutator bracket. A Lie-Yamaguti algebra is a binary-ternary algebra system which is denoted by ($\mathbb{G}$, [·,·], \{·,·,·\}) in this paper, and the denotion “LY-algebra” for “Lie-Yamaguti algebra” is used throughout as in \cite{ben2009}. The LY-algebras with the binary multiplication [·,·] = 0 are exactly the Lie triple systems (\cite{yam1957}-\cite{lis1952}), closely related with symmetric spaces, while the LY-algebras with the ternary multiplication \{·,·,·\} = 0 are the Lie algebras. Therefore, they can be considered as a simultaneous generalization of Lie triple systems and Lie algebras. Nomizu \cite{nom1954} characterized the reductive homogeneous spaces with canonical connection using some identities involving the torsion and the curvature. Then Yamaguti in \cite{yam1958} interpreted the above identities algebraically as a bilinear and a trilinear operations satisfying some axioms and called them “generalized Lie triple systems”. Later on, Kikkawa in \cite{kik1975} renamed these structures as Lie triple algebras and then he further studied about them. Recently, Kinyon and Weinstein in \cite{kin2001} introduced the notion of a LY-algebra in their study of Courant algebroids.

LY-algebras have been treated by several authors in connection with geometric problems on homogeneous spaces as studied in (\cite{sheng2021}-\cite{jha2015}). Their structure theories have been studied by Bentio et. al. in (\cite{ben2009, ben2005}). Zhou et. al. in \cite{zho2016} gave some basic properties including the relationship among the generalized derivation algebra, derivation algebra and quasi-derivation algebra of a Lie triple system. Recently, Fan and Chen in \cite{fan2023} gave some generalizations on ($\sigma,\tau$)-derivations of Lie superalgebras. Lin and Chen in \cite{lin2022} introduced the concept of quasi-derivations for LY-algebras. Inspired by the extensive research being carried out in the field of derivations related to Lie algebras and Lie triple systems separately, we are interested in generalizing the concept of $G$-derivations of LY-algebras which will simultaneously satisfy the properties of derivations of both the aforementioned algebraic structures. In this paper, we give the basic definitions for $G$-derivations and other generalized derivations of LY-algebars in section 2. Section 3 is devoted to the study of some important properties of $G$-derivations and its relation with other generalized derivations.

\medskip

More precisely, let $G$ be a subgroup of the automorphism group of a Lie algebra $L$, then a linear map $\Phi$ from $L$ to $L$ is said to be $G$-derivation if there exists $\theta,\vartheta \in G$ such that $$  \Phi([l_1 , l_2])=[\Phi(l_1), \theta(l_2)]+[\vartheta(l_1), \Phi(l_2)] ~~~~~~ \textnormal{for~all}~ l_1,l_2 \in L. $$

Here, this $G$-derivation reduces to a derivation in the classical sense if $\theta,\vartheta$ both are taken to be identity maps. In this article, we have generalized the notion of $G$-derivation to the case of LY-algebras and further studied some of its properties which are not trivial in general.\\

\section{Preliminaries}

\begin{definition}
A Lie-Yamaguti algebra (LY-algebra) is a vector space $\mathbb{G}$ over a field $\mathbb{K}$ with a bilinear operation $[\cdot,\cdot]:\mathbb{G}\times\mathbb{G}\rightarrow \mathbb{G}$ and a trilinear operation $\{\cdot,\cdot,\cdot\}:\mathbb{G}\times\mathbb{G}\times\mathbb{G}\rightarrow \mathbb{G}$ which satisfy:
\[[g,g]=0  \tag{LY1}\]  
\[\{g,g,h\}=0\tag{LY2}\]
\[\{g,h,i\}+\{h,i,g\}+\{i,g,h\}+[[g,h],i]+[[h,i],g]+[[i,g],h]=0 \tag{LY3}\]
\[\{[g,h],i,j\}+\{[h,i],g,j\}+\{[i,g],h,j\}=0 \tag{LY4}\]
\[\{g,h,[i,j]\}=[\{g,h,i\},j]+[i,\{g,h,j\}] \tag{LY5}\]
\[\{g,h,\{i,j,k\}\}=\{\{g,h,i\},j,k\}+\{i,\{g,h,j\},k\}+\{i,j,\{g,h,k\} \tag{LY6}\}\] for any $g,h,i,j,k \in\mathbb{G}.$ We denote this LY-algebra as the triple $(\mathbb{G},[\cdot,\cdot],\{\cdot,\cdot,\cdot\})$.
\end{definition}

If the binary bracket $[\cdot,\cdot]= 0$, then an LY-algebra reduces
to a Lie triple system. If the ternary bracket $\{\cdot,\cdot,\cdot\}= 0$, then an LY-algebra reduces to a Lie algebra.

If $(\mathbb{G},[\cdot,\cdot],\{\cdot,\cdot,\cdot\})$ be an LY-algebra, then a subalgebra $\mathbb{H}$ of $\mathbb{G}$ is a vector subspace $\mathbb{H}\subseteq \mathbb{G}$ satisfying $[\mathbb{H},\mathbb{H}]\subseteq \mathbb{H}$ and $\lbrace\mathbb{H},\mathbb{H},\mathbb{H}\rbrace\subseteq \mathbb{H}$. $\mathbb{H}$ is said to be an ideal of $\mathbb{G}$ if $\mathbb{H}\subseteq \mathbb{G}$ satisfying $[\mathbb{H},\mathbb{G}]\subseteq \mathbb{H}$ and $\lbrace\mathbb{H},\mathbb{G},\mathbb{G}\rbrace\subseteq \mathbb{H}$ (which implies that $[\mathbb{G},\mathbb{H}],\lbrace\mathbb{G},\mathbb{H},\mathbb{G}\rbrace,\lbrace\mathbb{G},\mathbb{G},\mathbb{H}\rbrace$ are also subsets of $\mathbb{H}$). An ideal $\mathbb{H}$ of $\mathbb{G}$ is said to be abelian if $[\mathbb{H},\mathbb{H}]=0$ and $\lbrace\mathbb{G},\mathbb{H},\mathbb{H}\rbrace=0$ (it follows that $\lbrace\mathbb{H},\mathbb{G},\mathbb{H}\rbrace=\lbrace\mathbb{H},\mathbb{H},\mathbb{G}\rbrace=0$). The subalgebra $[\mathbb{G},\mathbb{G}]+\lbrace\mathbb{G},\mathbb{G},\mathbb{G}\rbrace$ of $\mathbb{G}$ is called the derived algebra of $\mathbb{G}$, we denote it by $\mathbb{G}^{(1)}$. An LY-algebra $\mathbb{G}$ is perfect if $\mathbb{G}^{(1)}=\mathbb{G}$.

\medskip

\textbf{Example 1} : If on a left Leibniz algebra $(\mathbb{G},\cdot)$ we define $[g,h]=\dfrac{1}{2}(g\cdot h-h\cdot g)$ and $\{g,h,i\}=-\dfrac{1}{4}(g\cdot h)\cdot i$, then $(\mathbb{G},[\cdot,\cdot],\{\cdot,\cdot,\cdot\})$ is an LY-algebra.

\medskip

\textbf{Example 2} : Let $(\mathbb{G},[\cdot,\cdot])$ be a Lie algebra. We define $\{\cdot,\cdot,\cdot\}:\mathbb{G}\times\mathbb{G}\times\mathbb{G}\rightarrow \mathbb{G}$ by $\{g,h,i\}=[[g,h],i] ~\forall~ g,h,i\in \mathbb{G}$.
Then $(\mathbb{G},[\cdot,\cdot],\{\cdot,\cdot,\cdot\})$ becomes an LY-algebra naturally.

\begin{definition}
The center $Z(\mathbb{G})$ of an LY-algebra $\mathbb{G}$ can be defined as $Z(\mathbb{G})=\lbrace g\in\mathbb{G}~\vert~[g,h]=0~\textnormal{and}~\{g,h,i\}=\{h,i,g\}=0~\forall~ h,i\in\mathbb{G}\rbrace$. It is clear from the definition  that $\lbrace i,g,h\rbrace=0 $ for $g\in Z(\mathbb{G})$ and $\forall~ h,i\in \mathbb{G}$. We say that $\mathbb{G}$ is centerless if $Z(\mathbb{G})=\lbrace0\rbrace$.
\end{definition}

\begin{definition}
Let $(\mathbb{G},[\cdot,\cdot],\{\cdot,\cdot,\cdot\})$ and $(\mathbb{G}^\prime,[\cdot,\cdot]^\prime,\{\cdot,\cdot,\cdot\}^\prime)$ be two LY-algebras. A homomorphism from $(\mathbb{G},[\cdot,\cdot],\{\cdot,\cdot,\cdot\})$ to $(\mathbb{G}^\prime,[\cdot,\cdot]^\prime,\{\cdot,\cdot,\cdot\}^\prime)$ is a linear map $\Phi:\mathbb{G}\rightarrow\mathbb{G}^\prime$ such that for all $g,h,i\in \mathbb{G}$, $$\Phi([g,h])=[\Phi(g),\Phi(h)]^\prime,$$ $$\Phi(\lbrace g,h,i\rbrace)=\lbrace \Phi(g),\Phi(h),\Phi(i)\rbrace^\prime.$$
\end{definition}

\begin{definition}
Let $(\mathbb{G},[\cdot,\cdot],\{\cdot,\cdot,\cdot\})$ be an LY-algebra. A derivation on $\mathbb{G}$ is a linear map $D:\mathbb{G}\rightarrow \mathbb{G}$ such that 
$$D([g,h])=[D(g),h]+[g,D(h)],$$  $$D(\lbrace g,h,i \rbrace)=\lbrace D(g),h,i \rbrace+\lbrace g,D(h),i \rbrace+\lbrace g,h,D(i) \rbrace,$$  for all $g,h,i\in\mathbb{G} $.
\end{definition}

So, in case of LY-algebras a linear map is said to be a derivation if it is a derivation with respect to both the binary and the ternary operations of $\mathbb{G}$. The set of all derivations of $\mathbb{G}$ is denoted as $\textnormal{Der}(\mathbb{G})$, which forms a Lie algebra under the commutator bracket operation. Again, if we define the trilinear operation on $\textnormal{Der}(\mathbb{G})$ as in Example 2, then it becomes an LY-algebra.

\begin{definition}
A linear map $D:\mathbb{G}\rightarrow \mathbb{G} $ is  called  a quasi-derivation if there exist $D^\prime, D^{\prime\prime}\in \textnormal{End}(\mathbb{G})$ such that $$[D(g),h]+[g,D(h)]=D^\prime([g,h]),$$ $$\lbrace D(g),h,i\rbrace+\lbrace g,D(h),i\rbrace+\lbrace g,h,D(i)\rbrace=D^{\prime\prime}(\lbrace g,h,i\rbrace).$$
\end{definition}

For given an LY-algebra $(\mathbb{G},[\cdot,\cdot],\{\cdot,\cdot,\cdot\})$ and elements $g,h \in \mathbb{G}$, we define a mapping $D_{g,h}:\mathbb{G}\rightarrow\mathbb{G}$ such that $D_{g,h}(i)=\lbrace g,h,i\rbrace$. This map $D_{g,h}$ is a derivation (as can be shown using (LY5) and (LY6)) with respect to both the bilinear and trilinear operations. These derivations are known as inner derivations for LY-algebras.

Now we define the notion of $G$-derivation for LY-algebras.

\begin{definition}
Let G be a subgroup of $\textnormal{Aut}(\mathbb{G})$. A linear map $D:\mathbb{G}\rightarrow \mathbb{G} $ is  called  a  G-derivation of $\mathbb{G}$ if there exist two automorphisms $\theta ,\vartheta \in G$ such that $$D([g,h])=[D(g),\theta (h)]+[\vartheta (g),D(h)]$$ and $$D(\{g,h,i\})=\{D(g),\theta(h),\vartheta(i)\}+\{\vartheta(g),D(h),\theta(i)\}+\{\theta(g),\vartheta(h),D(i)\}$$ for all $g,h,i\in \mathbb{G}$. In this case $\theta ~and~ \vartheta$ are called the associated automorphisms of D.
\end{definition}

We write $\textnormal{Der}_G(\mathbb{G})$ for the set of all $G$-derivations of $\mathbb{G}$. Fix two automorphisms $\theta ,\vartheta \in G$, then by $\textnormal{Der}_{\theta ,\vartheta}(\mathbb{G})$ we denote the set of all $G$-derivations associated to $\theta $ and $ \vartheta$. Clearly, $\textnormal{Der}_{\theta,\vartheta}(\mathbb{G})\subseteq$ Der$_G(\mathbb{G})$ is a vector space and in particular, $\textnormal{Der}_{1,1}(\mathbb{G})=$ $\textnormal{Der}(\mathbb{G})$, where $\textnormal{Der}(\mathbb{G})$ is the set of all derivations of $\mathbb{G}$.

\begin{definition}
The set $C(\mathbb{G}) = \{D \in \textnormal{End}(\mathbb{G})~\mid~[D(g), h] = D([g, h]),~ and~ \{D(g), h, i\} = D(\{g, h, i\}),~ \forall~ g,h,i \in\mathbb{G}\} $
is called the centroid of $\mathbb{G}$.
\end{definition}
By (LY1)-(LY6), we can conclude that if $D\in C(\mathbb{G})$, then $D(\{g, h, i\})=\{D(g), h, i\} = \{g, D(h), i\} = \{g, h, D(i)\} $ and $D([g, h])=[D(g), h]=[g, D(h)]$ for all $g,h,i \in \mathbb{G}$.

\medskip

\section{Main Results}

\begin{proposition}\label{p1}
$\textnormal{dim}_\mathbb{K}(\textnormal{Der}_{\theta,\vartheta}(\mathbb{G}) )= \textnormal{dim}_\mathbb{K}(\textnormal{Der}_{\vartheta^{-1}\theta}(\mathbb{G}))$
\end{proposition}
\begin{proof}
We may define $f _ \vartheta : \textnormal{Der}_{\theta ,\vartheta}(\mathbb{G})\longrightarrow\textnormal{Der}_{\theta^{-1}\vartheta}(\mathbb{G})$ by
\begin{center}
$D\longmapsto\vartheta^{-1}\circ D$
\end{center}
In fact, for arbitrary $g,h,i\in \mathbb{G}$, since $D([g,h])=[D(g),\theta (h)]+[\vartheta (g),D(h)]$ and $$D(\{g,h,i\})=\{D(g),\theta(h),\vartheta(i)\}+\{\vartheta(g),D(h),\sigma(i)\}+\{\theta(g),\vartheta(h),D(i)\},$$
\\it follows that
\begin{align*}
\vartheta^{-1}\circ D([g,h])&=\vartheta^{-1}\circ([D(g),\theta(h)]+[\vartheta(g),D(h)])\\&=[\vartheta^{-1}\circ D(g),\vartheta^{-1}\theta(h)]+[g,\vartheta^{-1}\circ D(h)]
\end{align*}
and 
\begin{align*}
\vartheta^{-1}\circ D(\{g,h,i\})&=\vartheta^{-1}\circ(\{D(g),\theta(h),\vartheta(i)\}+\{\vartheta(g),D(h),\theta(i)\}+\{\theta(g),\vartheta(h),D(i)\})\\&=\{\vartheta^{-1}\circ D(g),\vartheta^{-1}\circ\theta(h),i\}+\{g,\vartheta^{-1}\circ D(h),\vartheta^{-1}\circ\theta(i)\}\\&~ +\{\vartheta^{-1}\circ\theta(g),h,\vartheta^{-1}\circ D(i)\}.
\end{align*}

So, from the above equations, we have $\vartheta^{-1}\circ D\in \textnormal{Der}_{\vartheta^{-1}\theta}(\mathbb{G})$. It can be shown that $\vartheta^{-1}\circ(D_1+D_2)=\vartheta^{-1}\circ D_1+\vartheta^{-1}\circ D_2$ and $\vartheta^{-1}\circ(a\cdot D)=a\cdot  \vartheta^{-1}\circ D$ $~\forall ~D,D_1,D_2\in\textnormal{Der}_{\theta,\vartheta}(\mathbb{G})$ and $a\in\mathbb{K}$. Thus $f_\vartheta$ is a linear map. Similarly, it can be shown that the map $g _ \vartheta : \textnormal{Der}_{\vartheta^{-1}\theta}(\mathbb{G})\longrightarrow\textnormal{Der}_{\theta,\vartheta}(\mathbb{G})$ defined by $D\longmapsto\vartheta\circ D$, is also linear, and $f_\vartheta^ {-1}=g_\vartheta$. Hence $g_\vartheta$ is a linear isomorphism and $\textnormal{dim}_\mathbb{F}(\textnormal{Der}_{\theta,\vartheta}(\mathbb{G}) )= \textnormal{dim}_\mathbb{F}(\textnormal{Der}_{\vartheta^{-1}\theta}(\mathbb{G}))$.
\end{proof}

The above isomorphism allows us to study  $\textnormal{Der}_{\vartheta^{-1}\theta}(\mathbb{G})$ rather than $\textnormal{Der}_{\theta,\vartheta}(\mathbb{G})$. Also, we can conclude that $\textnormal{Der}_{\theta, \theta}(\mathbb{G}) \cong \textnormal{Der}(\mathbb{G})$. In the next theorem we will see that $\textnormal{Der}_{\theta, \theta}(\mathbb{G})$ and $\textnormal{Der}(\mathbb{G})$ are isomorphic as LY-algebras. Let us define a  bracket $[\cdot ,\cdot]_\theta$ on $\textnormal{Der}_{\theta,\theta}(\mathbb{G})$ as follows
\begin{align*}
[D_1,D_2]_\theta =f_\theta^{-1}([f_\theta(D_1),f_\theta(D_2)])
\end{align*}
where $D_1, D_2\in\textnormal{Der}_{\theta,\theta}(\mathbb{G})$. With the above bracket, $\textnormal{Der}_{\theta,\theta}(\mathbb{G})$ forms a Lie algebra. Along with the above bracket if we define $\lbrace D_1,D_2,D_3\rbrace=[[D_1,D_2]_\theta,D_3]_\theta$, then it is easy to check that $\textnormal{Der}_{\theta,\theta}(\mathbb{G})$ forms an LY-algebra.

\begin{theorem}\label{t1}
 The LY-algebras  $\textnormal{Der}_{\theta,\theta}(\mathbb{G})$ and  $\textnormal{Der}(\mathbb{G})$ are isomorphic .
\end{theorem}
\begin{proof}
 From Proposition \ref{p1}, it only remains to show that $f_\theta$ is a LY-algebra homomorphism. Thus $f_\theta([D_1,D_2]_\theta)=f_\theta(f_\theta^{-1}([f_\theta(D_1),f_\theta(D_2)]))=[f_\theta(D_1),f_\theta(D_2)]$.
\end{proof}

\medskip
The vector space $\textnormal{Der}(\mathbb{G})\oplus\textnormal{Der}_\theta(\mathbb{G})$ can be made into an LY-algebra in the following way;

\begin{proposition}\label{p3}
Let $\theta$ be an involutive automorphism of $\mathbb{G}$ that commutes with each element of $\textnormal{Der}(\mathbb{G})$ and $\textnormal{Der}_\theta(\mathbb{G}).$ Then $\textnormal{Der}(\mathbb{G})\oplus\textnormal{Der}_\theta(\mathbb{G})$ is an LY-algebra.
\end{proposition}
\begin{proof}
Here it is obvious that $\textnormal{Der}(\mathbb{G})\oplus\textnormal{Der}_\theta(\mathbb{G})$ is a vector space. Thus it only remains to show that $\textnormal{Der}(\mathbb{G})\oplus\textnormal{Der}_\theta(\mathbb{G})$ is closed under both the bilinear and the trilinear bracket, i.e., we need to show that,
\begin{center}
$[D,T]=D\circ T-T\circ D\in \textnormal{Der}(\mathbb{G})~\textnormal{or}~\textnormal{Der}_\theta(\mathbb{G})~\forall~D\in \textnormal{Der}_{\theta^k}(\mathbb{G})~\textnormal{and}~T\in \textnormal{Der}_{\theta^l}(\mathbb{G})$
\end{center}
where $k,l\in$ \{0,1\}. Then it will be obvious to show by Example 2 that $\lbrace D,T,S\rbrace=[[D,T],S]\in \textnormal{Der}(\mathbb{G})\oplus\textnormal{Der}_\theta(\mathbb{G}).$ 
We will check this for both the bilinear and the trilinear operations.
\\For the bilinear operation, given arbitrary elements $g,h\in \mathbb{G}$, we have
\begin{align*}
D\circ T([g,h])&=D\circ([T(g),\theta^l(h)]+[g,T(h)]\\
&=[D\circ T(g),\theta^{k+l}(h)]+[T(g),D\circ \theta^l(h)]\\  &~+ \{[D(g),\theta^k\circ T(h)]+[g,D\circ T(h)]\},
\end{align*}
\vspace{-0.3cm}
and
\begin{align*}
T\circ D([g,h])&=T\circ([D(g),\theta^k(h)]+[g,D(h)])\\
&=[T\circ D(g),\theta^{l+k}(h)]+[D(g),T\circ\theta^k(h)]\\ &
~+\{[T(g),\theta^l\circ D(h)]+[g,T\circ D(h)]\}.
\end{align*}

Since $\theta$ commutes with every element of $\textnormal{Der}(\mathbb{G})$ and $\textnormal{Der}_\theta(\mathbb{G})$, thus
\begin{align}
(D\circ T-T\circ D)([g,h]) &=[[D,T](g),\theta^{k+l}(h)]+[T(g),(D\circ \theta^l-\theta^l\circ D)(h)]\nonumber\\ &~~+[D(g),(\theta^k\circ T-T\circ\theta^k)(h)]+[g,[D,T](h)]\nonumber\\ &=[[D,T](g),\theta^{k+l}(h)]+[g,[D,T](h)].
\end{align}
Now for the trilinear operation, given arbitrary elements $g,h,i\in \mathbb{G}$, we have
\begin{align*}
D\circ T(\lbrace g,h,i\rbrace)&=D\circ(\{T(g),\theta^l(h),i\}+\{g,T(h),\theta^l(i)\}+\{\theta^l(g),h,T(i)\})\\
&=\{D\circ T(g),\theta^{k+l}(h),i\}+\{T(g),D\circ \theta^l(h),\theta^k(i)\}+\{\theta^k\circ T(g),\theta^l(h),D (i)\}
\\&~+\{D(g),\theta^k\circ T(h),\theta^l(i)\}+\{g,D\circ T(h),\theta^{k+l}(i)\}+\{\theta^k (g),T(h),D\circ\theta^l(i)\}
\\&~+\{D\circ\theta^l(g),\theta^k(h),T(i)\}+\{\theta^l(g),D(h),\theta^k\circ T(i)\}+\{\theta^{k+l}(g),h,D\circ T(i)\},
\end{align*}
\begin{align*}
T\circ D(\lbrace g,h,i\rbrace)&=T\circ(\{D(g),\theta^k(h),i\}+\{g,D(h),\theta^k(i)\}+\{\theta^k(g),h,D(i)\})\\
&=\{T\circ D(g),\theta^{l+k}(h),i\}+\{D(g),T\circ \theta^k(h),\theta^l(i)\}+\{\theta^l\circ D(g),\theta^k(h),T (i)\}
\\&~+\{T(g),\theta^l\circ D(h),\theta^k(i)\}+\{g,T\circ D(h),\theta^{l+k}(i)\}+\{\theta^l (g),D(h),T\circ\theta^k(i)\}
\\&~+\{T\circ\theta^k(g),\theta^l(h),D(i)\}+\{\theta^k(g),T(h),\theta^l\circ D(i)\}+\{\theta^{l+k}(g),h,T\circ D(i)\}.
\end{align*}
It is clear from the above that 
\begin{align}
(D\circ T-T\circ D)(\lbrace g,h,i\rbrace) &=\{[D,T](g),\theta^{k+l}(h),i\}+\{g,[D,T](h),\theta^{k+l}(i)\}
\nonumber\\&~+\{\theta^{k+l}(g),h,[D,T](i)\}.
\end{align}

As $\theta^{k+l}=\theta~ \textnormal{or}~I~$, it follows from equations (2.5) and (2.6) that $[D,T]$ is either in $\textnormal{Der}(\mathbb{G})~\textnormal{or}~\textnormal{Der}_\theta(\mathbb{G}) $. Hence, $\textnormal{Der}(\mathbb{G})\oplus\textnormal{Der}_\theta(\mathbb{G})~$ is an LY-algebra.
\end{proof}

\medskip

\begin{proposition}\label{p4}
Let $\mathbb{G}$ be a non-abelian LY-algebra and $D\in\textnormal{Der}_\theta(\mathbb{G})$ be an element such that $[D,\theta](\mathbb{G})\subseteq Z(\mathbb{G})$. Then the derived algebra $\mathbb{G}^{(1)}$ of $\mathbb{G}$ is contained in the kernel of $[D,\theta]$. In this case, if $\mathbb{G}$ is perfect, then $D$ commutes with $\theta$.
\end{proposition}
\begin{proof}
Consider $g,h,i\in \mathbb{G}$, then we have 
\begin{align*}
(\theta\circ D)([g,h])&=\theta([D(g),\theta(h)]+[g,D(h)])\\
& =[\theta\circ D(g),\theta^2 (h)]+[\theta(g),\theta\circ D(h)],
\end{align*}
and
\begin{align*}
(D\circ\theta)([g,h])&=D([\theta(g),\theta(h)])\\ &=[D\circ\theta(g),\theta^2(h)]+[\theta(g),D\circ\theta(h)].     
\end{align*}
Since $[D,\theta](\mathbb{G})\subseteq Z(\mathbb{G})$, we will have
\begin{align*}
[D,\theta]([g,h])&=(D\circ\theta-\theta\circ D)([g,h])\\ &=[[D,\theta](g),\theta^2(h)]+[\theta(g),[D,\theta](h)]    \\&=0.
\end{align*}

Similarly,
\begin{align*}
(D\circ\theta)(\lbrace g,h,i\rbrace) &=D\circ (\lbrace \theta(g),\theta(h),\theta(i)\rbrace) \\&=\lbrace D\circ\theta(g),\theta^2(h),\theta(i)\rbrace+\lbrace \theta(g),D\circ\theta(h),\theta^2(i)\rbrace+\lbrace \theta^2(g),\theta(h),D\circ\theta(i)\rbrace,
\end{align*}
and
\begin{align*}
(\theta\circ D)(\lbrace g,h,i\rbrace) &=\theta\circ (\lbrace D(g),\theta(h),i\rbrace+\lbrace g,D(h),\theta(i)\rbrace+\lbrace \theta(g),h,D(i)\rbrace) \\&=\lbrace\theta\circ D(g),\theta^2(h),\theta(i)\rbrace+\lbrace \theta(g),\theta\circ D(h),\theta^2(i)\rbrace+\lbrace \theta^2(g),\theta(h),\theta\circ D(i)\rbrace.     
\end{align*}

Again since $[D,\theta](\mathbb{G})\subseteq Z(\mathbb{G})$, we have
\begin{align*}
[D,\theta](\lbrace g,h,i\rbrace) =0.
\end{align*}
Therefore, $\mathbb{G}^{(1)} \subseteq \ker [D,\theta]$. Now, as $\mathbb{G}$ is perfect, then $\mathbb{G} \subseteq\textnormal{ker} [D,\theta]$. Thus $\textnormal{ker}([D,\theta])=\mathbb{G}$, i.e., $[D,\theta]=0$.
\end{proof}

\medskip

\begin{proposition}
Suppose $\theta\in G$ and $D\in C(\mathbb{G})\cap\textnormal{Der}_\theta(\mathbb{G})$. Then ${D(i)}\in \textnormal{ker}(D_{g,h})~for ~all~g,h,i\in \mathbb{G}$. In particular, if $\mathbb{G}$ is centerless, then $C(\mathbb{G})\cap\textnormal{Der}_\theta(\mathbb{G})=0$.
\end{proposition}
\begin{proof}
Given $D\in C(\mathbb{G})\cap\textnormal{Der}_\theta(\mathbb{G})$, which implies $\lbrace \theta(g),h,D(i)\rbrace=0$. But here $\theta$ being a bijection, we get $\lbrace g,h,D(i)\rbrace=0$ for all $g,h,i\in \mathbb{G}$  i.e., ${D(i)}\in \textnormal{ker}(D_{g,h})$. Again if $Z(\mathbb{G})$ is zero, then the adjoint map will be injective and will have a left inverse. Therefore, we can conclude that $\theta=0$, i.e., $C(\mathbb{G})\cap\textnormal{Der}_\theta(\mathbb{G})=0$.  
\end{proof}
\medskip

Let $\mathbb{G}$ be an LY algebra and $\mathbb{H}$ be a subalgebra of $\mathbb{G}$. Consider an automorphism $\theta$ of $\mathbb{G}$ such that $\mathbb{H}$ is invariant under $\theta$, i.e., $\theta(\mathbb{H})\subseteq \mathbb{H}$. Then we denote $\textnormal{Der}_{\theta,\mathbb{H}}(\mathbb{G})$ as the set of all $\theta$-derivations of $\mathbb{G}$ that stabilize $\mathbb{H}$, i.e.,

$$\textnormal{Der}_{\theta,\mathbb{H}}(\mathbb{G})=\{D\in\textnormal{Der}_\theta(\mathbb{G})\mid D(\mathbb{H})\subseteq \mathbb{H}\}.$$

\medskip

\begin{proposition}
$\textnormal{Der}_{\theta,\mathbb{H}}(\mathbb{G})$ is a subspace of $\textnormal{Der}_\theta(\mathbb{G})$. Moreover, if $\mathbb{H}$ is a perfect ideal of $\mathbb{G}$, then $\textnormal{Der}_{\theta,\mathbb{H}}(\mathbb{G})=\textnormal{Der}_\theta(\mathbb{G})$.                                                                                                                                                                    
\end{proposition}
\begin{proof}
Here it is obvious that $\textnormal{Der}_{\theta,\mathbb{H}}(\mathbb{G})$ is a subset of $\textnormal{Der}_\theta(\mathbb{G})$. So it only remains to show that $\textnormal{Der}_{\theta,\mathbb{H}}(\mathbb{G})$ is a vector space itself. Now for $a,b\in \mathbb{K}$ and $D,T\in\textnormal{Der}_{\theta,\mathbb{H}}(\mathbb{G})$, we have $(aD+bT)(h)=aD(h)+bT(h)\in \mathbb{H}$ for all $h\in \mathbb{H}$. Since, $h$ was chosen arbitrarily, so $\textnormal{Der}_{\theta,\mathbb{H}}(\mathbb{G})$ is a subspace of $\textnormal{Der}_\theta(\mathbb{G})$. Now to prove the second part, it is enough to show $\textnormal{Der}_\theta(\mathbb{G})\subseteq \textnormal{Der}_{\theta,\mathbb{H}}(\mathbb{G})$. Let $D\in \textnormal{Der}_\theta(\mathbb{G})$ and $h\in \mathbb{H}$. Since $\mathbb{H}$ is perfect, so $\mathbb{H}^{(1)}=[\mathbb{H},\mathbb{H}]+\lbrace\mathbb{H},\mathbb{H},\mathbb{H}\rbrace=\mathbb{H}$. Now for any $h\in \mathbb{H}$, we can write $h= \sum_{i=1}^{n} [h_{1i},h_{2i}]+\sum_{i=1}^{n} \lbrace h_{3i},h_{4i},h_{5i}\rbrace$ for some $h_{1i},h_{2i},h_{3i},h_{4i},h_{5i} \in \mathbb{H}$; $1\leqslant i \leqslant n$. Since $\theta$  stabilizes $\mathbb{H}$, so we have,
\begin{align*}
D(h)&=\sum_{i=1}^{n} D([h_{1i},h_{2i}])+\sum_{i=1}^{n} D(\lbrace h_{3i},h_{4i},h_{5i}\rbrace)\\&=\sum_{i=1}^{n} ([D(h_{1i}),\theta(h_{2i})]+[h_{1i},D(h_{2i})])\\&~+\sum_{i=1}^{n}(\lbrace D(h_{3i}),\theta(h_{4i}),h_{5i}\rbrace+\lbrace h_{3i},D(h_{4i}),\theta(h_{5i})\rbrace+\lbrace \theta(h_{3i}),h_{4i},D(h_{5i})\rbrace) .
\end{align*} 

Now as $\mathbb{H}$ is an ideal, so $D(h)\in \mathbb{H}$ i.e., $D(\mathbb{H})\subseteq \mathbb{H}$ which asserts that $\textnormal{Der}_\theta(\mathbb{G})\subseteq \textnormal{Der}_{\theta,\mathbb{H}}(\mathbb{G})$. Therefore, $\textnormal{Der}_{\theta,H}(\mathbb{G})=\textnormal{Der}_\theta(\mathbb{G})$.                                                                                                                                                                    
\end{proof}

Since $\mathbb{H}$ is $\theta$-stable, so the automorphism $\theta$ restricts to an automorphism $\theta\mid_\mathbb{H}$ of $\mathbb{H}$ and also an arbitrary element $D\in\textnormal{Der}_{\theta,\mathbb{H}}(\mathbb{G})$ can be restricted to a $\theta\mid_\mathbb{H}$ -derivation of $\mathbb{H}$. This restriction map induces a natural linear map:
$$\delta:\textnormal{Der}_{\theta,\mathbb{H}}(\mathbb{G})\longrightarrow\textnormal{Der}_{\theta\mid_\mathbb{H}}(\mathbb{H})~~~~ {\rm{given~ by}}~~D\mapsto D\mid_\mathbb{H}.$$

For $D\in\textnormal{Der}_{\theta,\mathbb{H}}(\mathbb{G})$, we define a map $\overset{\frown}{D}:\mathbb{G}\longrightarrow \mathbb{G}$ by

$$\overset{\frown}{D}([g,h])= 2D([g,h])+[D(h),\theta(g)]+[\theta(h),D(g)],$$ and  $$\overset{\frown}{D}(\lbrace g,h,i\rbrace)=3D(\lbrace g,h,i\rbrace)+\lbrace D(g),\theta(h),i\rbrace+\lbrace g,D(h),\theta(i)\rbrace+\lbrace \theta(g),h,D(i)\rbrace.$$

Now we give some properties of the above map.

\begin{proposition}
Suppose the inner derivation $D_{g,h}\in \textnormal{Aut}(\mathbb{H})$ for some $g,h\in \mathbb{G}$. Then $\overset{\frown}{D}$ restricts to a homogeneous linear map $\overset{\frown}{D}\mid_H:H\longrightarrow H$.
\end{proposition}
\begin{proof}
First observe that $\mathbb{H}=D_{g,h}(\mathbb{H})=\lbrace g,h,\mathbb{H}\rbrace\subseteq\lbrace\mathbb{G},\mathbb{G},\mathbb{G}\rbrace$. Thus for any $x\in \mathbb{H}$  there exists $y\in \mathbb{H}$ such that $x=\lbrace g,h,y\rbrace$  and $\overset{\frown}{D}(x)=\overset{\frown}{D}(\lbrace g,h,y\rbrace)=3D(\lbrace g,h,y\rbrace)+\lbrace D(g),\theta(h),y\rbrace+\lbrace g,D(h),\theta(y)\rbrace +\lbrace \theta(g),h,D(y)\rbrace \in D(\mathbb{H})+[\mathbb{G},\mathbb{G},\mathbb{H}]\subseteq \mathbb{H}$. Thus, $\overset{\frown}{D}(\mathbb{H})\subseteq \mathbb{H}$.  Now to check the map $\overset{\frown}{D}\mid_\mathbb{H}$ is linear, we take another element $x^\prime\in \mathbb{H}$ with $x^\prime=\lbrace g,h,y^\prime\rbrace$ for some $ y^\prime\in \mathbb{H}$. Then for $a,b\in\mathbb{F}$,
\begin{align*}
\overset{\frown}{D}(ax+bx^\prime) &=\overset{\frown}{D}(\lbrace g,h,ay\rbrace+\lbrace g,h,by^\prime\rbrace)
\\ &=\overset{\frown}{D}(\lbrace g,h,ay+by^\prime\rbrace)
\\ &=3D(\lbrace g,h,ay+by^\prime\rbrace)+\lbrace D(g),\theta(h),ay+by^\prime\rbrace+\lbrace g,D(h),\theta(ay+by^\prime)\rbrace
\\ &~+\lbrace \theta(g),h,D(ay+by^\prime)\rbrace
\\ &=3D(\lbrace g,h,ay\rbrace)+\lbrace D(g),\theta(h),ay\rbrace+\lbrace g,D(h),\theta(ay)\rbrace+\lbrace \theta(g),h,D(ay)\rbrace
\\ &~+3D(\lbrace g,h,by^\prime\rbrace)+\lbrace D(g),\theta(h),by^\prime\rbrace+\lbrace g,D(h),\theta(by^\prime)\rbrace+\lbrace \theta(g),h,D(by^\prime)\rbrace
\\ &=a\{3D(\lbrace g,h,y\rbrace)+\lbrace D(g),\theta(h),y\rbrace+\lbrace g,D(h),\theta(y)\rbrace+\lbrace \theta(g),h,D(y)\rbrace\}
\\ &~+b\{3D(\lbrace g,h,y^\prime\rbrace)+\lbrace D(g),\theta(h),y^\prime\rbrace+\lbrace g,D(h),\theta(y^\prime)\rbrace+\lbrace \theta(g),h,D(y^\prime)\rbrace\}
\\&=a\overset{\frown}{D}(x)+b\overset{\frown}{D}(x^\prime).
\end{align*}
 
Thus  $\overset{\frown}{D}$ is a linear map.
\end{proof}

\medskip

Let us denote the set of all quasi-derivations of $\mathbb{H}$ by $\textnormal{Qder}(\mathbb{H})$, then we have the following result.

\begin{proposition}
Suppose $D(g_1,g_2)\in \textnormal{Aut}(\mathbb{H})$ for some $g_1,g_2\in \mathbb{G}$ with $\theta(g_1)=g_1$ and $D(g_1),D(g_2)\in Z(\mathbb{G}) $ for all $D\in\textnormal{Der}_{\theta,\mathbb{H}}(\mathbb{G})$. Then $\delta(\textnormal{Der}_{\theta,\mathbb{H}}(\mathbb{G}))\subseteq\textnormal{Qder}(\mathbb{H})$. 
\end{proposition}
\begin{proof}
Let $h_1$ and $h_2$ be any arbitrary elements in $\mathbb{H}$, and $D\in \textnormal{Der}_{\theta,\mathbb{H}}(\mathbb{G})$. Since $D(g_1),D(g_2)\in Z(\mathbb{G})$, we have 
\begin{align*}
D\circ D(g_1,g_2)(h_1)  &=~D(\lbrace g_1,g_2,h_1\rbrace)    \\&=~\lbrace D(g_1),\theta(g_2),h_1\rbrace+\lbrace g_1,D(g_2),\theta(h_1)\rbrace+\lbrace \theta(g_1),g_2,D(h_1)\rbrace     \\&=~\lbrace\theta(g_1),g_2,D(h_1)\rbrace
\\&=~ \lbrace g_1,g_2,D(h_1)\rbrace \\&=~D(g_1,g_2)\circ D(h_1)
\end{align*} 


As $D(g_1,g_2)\in \textnormal{Aut}(\mathbb{H})$, we have 
\begin{align*}
\overset{\frown}{D} \mid_h\circ~ D(g_1,g_2)
([h_1,h_2])      &=\overset{\frown}{D}([D(g_1,g_2)(h_1),D(g_1,g_2)(h_2)])     
\\&=\overset{\frown}{D}([\lbrace g_1,g_2,h_1\rbrace,\lbrace g_1,g_2,h_2\rbrace])
\\&=2D([\lbrace g_1,g_2,h_1\rbrace,\lbrace g_1,g_2,h_2\rbrace])+[D(\lbrace g_1,g_2,h_2\rbrace),\theta(\lbrace g_1,g_2,h_1\rbrace)]
\\&~+[\theta(\lbrace g_1,g_2,h_2\rbrace),D(\lbrace g_1,g_2,h_1\rbrace)]
\\&=D([\lbrace g_1,g_2,h_1\rbrace,\lbrace g_1,g_2,h_2\rbrace])-[D(\lbrace g_1,g_2,h_1\rbrace),\theta(\lbrace g_1,g_2,h_2\rbrace)]
\\&~-D([\lbrace g_1,g_2,h_2\rbrace],\lbrace g_1,g_2,h_1\rbrace)+[D(\lbrace g_1,g_2,h_2\rbrace),\theta(\lbrace g_1,g_2,h_1\rbrace)]
\\&=[\lbrace g_1,g_2,h_1\rbrace,D(\lbrace g_1,g_2,h_2\rbrace)]+[D(\lbrace g_1,g_2,h_1\rbrace),(\lbrace g_1,g_2,h_2\rbrace)]
\\&=[D\circ D(g_1,g_2)(h_1),D(g_1,g_2)(h_2)]+[D(g_1,g_2)(h_1),D\circ D(g_1,g_2)(h_2)]
\\&=D(g_1,g_2)([D(h_1),h_2]+[h_1,D(h_2)])
\end{align*}
Thus, $[D(h_1),h_2]+[h_1,D(h_2)]=D(g_1,g_2)\circ\overset{\frown}{D}\mid_h\circ D(g_1,g_2)([h_1,h_2]).$ Clearly, $D$ restricts to a quasi-derivation of $\mathbb{H}$ and we obtain $\delta(\textnormal{Der}_{\theta,\mathbb{H}}(\mathbb{G}))\subseteq\textnormal{Qder}(\mathbb{H}).$
\end{proof}

\medskip

\end{document}